\documentclass[letterpaper,11pt]{amsart}

\usepackage{graphicx,color,epsfig}
\usepackage{amsfonts,amsmath,amssymb,amsthm}
\usepackage{hyperref}
\usepackage{comment}
\usepackage{mathtools}

\newtheorem{thm}{Theorem}[section]
\newtheorem{prop}[thm]{Proposition}
\newtheorem{lem}[thm]{Lemma}

\newtheorem{con}[thm]{Conjecture}
\newtheorem{asm}{Assumption}

\theoremstyle{remark}
\newtheorem{rem}[thm]{Remark}

\theoremstyle{definition}

\newcommand{\ra}{\rightarrow}

\newcommand{\Z}{\mathbb Z}     

\renewcommand{\d}{\delta}

\renewcommand{\k}{\kappa}

\newcommand{\e}{\varepsilon}

\newcommand{\s}{\sigma}

\newcommand{\ind}[1]{ \mathbf{1}_{ \{ #1 \} } } 

\newcommand{\w}{\omega}              
\renewcommand{\P}{\mathbb{P}}        
\newcommand{\E}{\mathbb{E}}          

\newcommand{\be}{\begin{equation}}
\newcommand{\ee}{\end{equation}}

\def\Pv{\mathbf{P}}  \def\Ev{\mathbf{E}}   

\title[Strong transience for RWRE]{Strong transience of one-dimensional random walk in a random environment}
\author{Jonathon Peterson}
\address{Jonathon Peterson\\Purdue University\\Department of Mathematics\\150 N University Street\\West Lafayette, IN  47907\\USA}
\email{peterson@math.purdue.edu}
\urladdr{http://www.math.purdue.edu/~peterson}
\thanks{J. Peterson was partially supported by NSA grants H98230-13-1-0266 and H98230-15-1-0049.}


\date{\today}

\begin{document}

\begin{abstract}
A transient stochastic process is considered \emph{strongly transient} if conditioned on returning to the starting location, the expected time it takes to return the the starting location is finite. 
We characterize strong transience for a one-dimensional random walk in a random environment. 
We show that under the quenched measure transience is equivalent to strong transience, while under the averaged measure strong transience is equivalent to ballisticity (transience with non-zero limiting speed).  
\end{abstract}

\maketitle

\section{Introduction and statement of main results}

The notions of transience and recurrence of stochastic processes are well known, but somewhat less well known is the notion of \emph{strong transience}. 
Let $\{Z_n\}_{n\geq 0}$ be a stochastic process on some countable state space, and let $ \mathfrak{R} = \inf\{ n \geq 1: Z_n = Z_0 \}$ be the first time that the process returns to its initial location. 
The process $Z_n$ is said to be stongly transient if it is transient and $\Ev[\mathfrak{R} | \, \mathfrak{R} < \infty] < \infty$. 
(This is a ``strong'' notion of transience since it implies that any returns to the starting point must happen relatively quickly.) 
If $Z_n$ is transient but $\Ev[ \mathfrak{R}| \, \mathfrak{R}<\infty] = \infty$, then we will say that the process is weakly transient. 
In this paper we will consider one-dimensional random walks in a random environment (RWRE) and will give a simple characterization of strong transience when the distribution of the environment is an i.i.d.\ product measure. 
Our main results show that the characterization of strong transience is different under the quenched and averaged measures. 
Under the quenched measure, we will show that strong transience is equivalent to transience, while under the averaged measure strong transience is equivalent to transience with non-zero limiting speed. 

The question of strong transience for RWRE was posed by Kosygina and Zerner as Problem 1.6 in \cite{kzEERW} where they also studied strong transience of one-dimensional excited random walks. 
While the models of RWRE and excited random walks are very different models of self-interacting random motions, there is a remarkable similarity in many of the results in the two models. In particular, (under the averaged measures) the limiting distributions for transient RWRE \cite{kksStable} are very similar to those for transient excited random walks \cite{bsRGCRW,kzPNERW,kmLLCRW}. In both models, the limiting distributions show three distinct types of behavior that can occur depending on the particular parameters of the model: 1) transience with sublinear speed and non-Gaussian limiting distributions, 2) transience with non-zero speed and non-Gaussian limiting distributions, and 3) transience with non-zero speed and Gaussian limiting distributions. 
In \cite{kzEERW} it was shown that within the second regime (non-zero limiting speed and non-Gaussian limiting distributions) there is a transition from weak to strong transience. 
Our results, however, show that (under the averaged measure) the transition from weak to strong transience coincides with the transition from sublinear speed to non-zero limiting speed.

\subsection{One-dimensional RWRE}
An \emph{environment} for a one-dimensional RWRE is a sequence $\w = \{\w_x\}_{x\in \Z} \in [0,1]^\Z$. 
Given an environment $\w$ and a fixed $z\in\Z$, the random walk $\{X_n\}_{n\geq 0}$ started at $z$ in the environment $\w$ is the Markov chain with law $P_\w^z$ given by $P_\w^z(X_0 = z) = 1$ and
\[
 P_\w^z( X_{n+1} = y | \, X_n = x ) = \begin{cases} \w_x & y = x+1 \\ 1-\w_x & y = x-1 \\ 0 & \text{otherwise}. \end{cases}
\]
For random walks in random environments, we also let the environment $\w$ be chosen randomly. In this paper we will make the following assumption on the randomness of the environment. 
\begin{asm}\label{asmiid}
 The distribution $P$ on environments is such that $\w = \{\w_x\}_{x\in \Z}$ is an i.i.d.\ sequence. 
\end{asm}
The distribution $P_\w^z$ of the random walk for a fixed environment $\w$ is called the \emph{quenched} law of the random walk. 
By averaging the quenched law with respect to the distribution $P$ on environments 
we obtain what is called the \emph{averaged} (or \emph{annealed}) law 
\[
 \P^z(\cdot) = E_P[ P_\w^z(\cdot) ]. 
\]
Here $E_P$ denotes the expectation with respect to the measure $P$ on environments. Expectations with respect to the quenched and averaged measures on the random walk will be denoted by $E_\w^z$ and $\E^z$, respectively. 
It will often be the case that we will be interested in the RWRE started at $X_0 = 0$, and thus we will use the notation $P_\w$ and $\P$ to denote $P_\w^0$ and $\P^0$, respectively (corresponding expectations will be denoted $E_\w$ and $\E$.)

The study of RWRE was initiated in Solomon's seminar paper \cite{sRWRE}. 
In this paper, Solomon gave a characterization of recurrence/transience of one-dimensional RWRE and also calculated the limiting speed. 
Before stating Solomon's results, we first introduce some notation. 
Let
\begin{equation}\label{rhodef}
 \rho_x = \frac{1-\w_x}{\w_x}, \quad \text{for } x \in \Z. 
\end{equation}
With this notation, Solomon's results can be stated as follows. 
\begin{thm}[\cite{sRWRE}]\label{thm:sol}
Assume that the distribution on environments $P$ satisfies Assumption \ref{asmiid}, and assume that $E_P[\log \rho_0]$ exists. 
\begin{enumerate}
 \item The recurrence or transience of the RWRE is determined by the value of $E_P[\log \rho_0]$. 
\begin{itemize}
 \item If $E_P[\log \rho_0] < 0$ then $\P( \lim_{n\ra\infty} X_n = +\infty) = 1$. 
 \item If $E_P[\log \rho_0] > 0$ then $\P( \lim_{n\ra\infty} X_n = -\infty) = 1$.
 \item If $E_P[\log \rho_0] = 0$ then $\P( \liminf_{n\ra\infty} X_n = -\infty \text{ and } \limsup_{n\ra\infty} X_n = +\infty) = 1$.
\end{itemize}
 \item The limiting speed of the RWRE is determined by the values of $E_P[\rho_0]$ and $E_P[\rho_0^{-1}]$. In particular,
\begin{equation}\label{speedform}
 \lim_{n\ra\infty} \frac{X_n}{n} 
=\begin{cases}
  \frac{1-E_P[\rho_0]}{1+E_P[\rho_0]} & \text{if } E_P[\rho_0] < 1\\
-  \frac{1-E_P[\rho_0^{-1}]}{1+E_P[\rho_0^{-1}]} & \text{if } E_P[\rho_0^{-1}] < 1 \\
  0 & \text{if } E_P[\rho_0^{-1}], \, E_P[\rho_0] \geq 1,
 \end{cases}
\quad \P\text{-a.s.}
\end{equation}
\end{enumerate}
\end{thm}
\begin{rem}
 Note that Jensen's inequality implies that $1/E_P[\rho_0^{-1}] \leq E_P[\rho_0]$, so the formula for the speed in \eqref{speedform} covers the three possible cases. 
\end{rem}

We are now ready to state the main result of the paper. For simplicity, we will state our results for RWRE that are transient to the right. 
\begin{thm}\label{thm:st}
Assume that the distribution on environments $P$ satisfies Assumption \ref{asmiid} and that $E_P[\log \rho_0] \in (-\infty,0)$. Then
 \begin{enumerate}
 \item \label{qst} $E_\w[  \mathfrak{R}| \,  \mathfrak{R}<\infty] < \infty $ for $P$-a.e.\ environment $\w$. 
 \item \label{ast} $\E[ \mathfrak{R} | \,  \mathfrak{R}<\infty]  < \infty \iff E_P[\rho_0] < 1$.  
 \end{enumerate}
\end{thm}

Clearly strong transience under the averaged measure requires both $\E[ \mathfrak{R} | \, X_1=-1, \,  \mathfrak{R}<\infty]< \infty$ and  $\E[ \mathfrak{R} | \, X_1=1, \,  \mathfrak{R}<\infty]< \infty$. If the random walk is transient to the right, then it is the second of these conditional expectations that is more interesting. 
In the proof of Theorem \ref{thm:st} we will show that $\E[ \mathfrak{R} | \, X_1=-1, \,  \mathfrak{R}<\infty]< \infty \iff E_P[\rho_0] < 1$, but for the other conditional expectation we only need that $E_P[\rho_0]<1$ implies that $\E[ \mathfrak{R} | \, X_1=1, \, \mathfrak{R}<\infty]< \infty$. 
The next theorem gives the converse of this last statement under a slightly stronger assumption on the environment. 

\begin{thm}\label{thm:srt}
 Assume that the distribution on environments $P$ satisfies Assumption \ref{asmiid} and that $E_P[\log \rho_0] \in (-\infty,0)$.
\begin{enumerate}
 \item\label{srtp1} If $E_P[\rho_0] < 1$ then $\E[ \mathfrak{R} | \, X_1=1, \,  \mathfrak{R}<\infty]< \infty$.
 \item\label{srtp2} If either $E_P[\rho_0] > 1$ or $E_P[\rho_0] = 1$ and $E_P[\rho_0 \log \rho_0] < \infty$ then $\E[ \mathfrak{R} | \, X_1=1, \,  \mathfrak{R}<\infty] = \infty$. 
\end{enumerate}
\end{thm}

The assumptions in part \ref{srtp2} of Theorem \ref{thm:srt} are only slightly stronger than $E_P[\rho_0 ] \geq 1$. We conjecture, however, that the result is true under this weaker assumption as well. 
\begin{con}
 If $P$ satisfies Assumption \ref{asmiid} and $E_P[\log\rho_0] \in (-\infty,0)$, then 
\[
 \E[ \mathfrak{R} | \, X_1=1, \,  \mathfrak{R}<\infty]< \infty \quad \iff \quad E_P[\rho_0] < 1. 
\]
\end{con}

\section{General random walk results}\label{sec:GRW}

In this section we record some general results on one-dimensional random walks that will be useful for analyzing the environment $\w$. 
Assume that $\xi_1,\xi_2,\ldots$ is an i.i.d.\ sequence of random variables
and let $S_n = \sum_{i=1}^n \xi_i$ for any $n\geq 1$. 
To avoid confusion with the probability measures associated to the RWRE, we will use $\Pv$ for the law of the sequence $(\xi_1,\xi_2,\ldots)$ and $\Ev$ for corresponding expectations. 
We will always assume that $\Ev[\xi_1] \in (-\infty,0)$ so that the random walk $S_n$ has negative drift. 
The following result gives asymptotics for the probability that the random walk $S_n$ goes above level $t\geq 0$ at some point.  
\begin{prop}\label{RWhp}
 Assume that $\Ev[\xi_1] < 0$ and that $\Ev[ e^{\gamma \xi_1}] = 1$ and $\Ev[ \xi_1 e^{\gamma \xi_1} ] < \infty$ for some $\gamma > 0$. 
\begin{enumerate}
 \item\label{p1} If the distribution of $\xi_1$ is non-lattice then there exists a constant $C>0$ such that 
\[
 \lim_{t \ra\infty} e^{\gamma t}\, \Pv\left( \sup_{n\geq 1} S_n \geq t \right) = C. 
\]
 \item\label{p2} If $\Pv( \xi_1 \in a \Z)$ for some $a > 0$, then there exists a constant $C>0$ such that 
\[
 \lim_{k \ra\infty} e^{\gamma k a} \, \Pv\left( \sup_{n\geq 1} S_n \geq k a \right) = C.
\]
\end{enumerate}
\end{prop}

\begin{rem}\label{remlattice}
 Part \ref{p1} of Proposition \ref{RWhp} is the content of \cite[Lemma 1]{iEVQ}. 
The proof of the lattice case in part \ref{p2} is essentially the same as the proof of the non-lattice case in \cite{iEVQ}, but we will include the proof here for completeness since the proof is short. 
\end{rem}

\begin{proof}[Proof of part \ref{p2} of Proposition \ref{RWhp}]
%
The key to the proof of the proposition is the following change of measure. Let $Q$ be a measure on sequences $(\xi_1,\xi_2,\ldots)$ with Radon-Nykodym derivative given by 
\[
 \frac{dQ}{d\Pv}(\xi_1,\ldots,\xi_n) = e^{\gamma S_n}. 
\]
Expectations with respect to the measure will be denoted by $E_Q$. 
Note that $Q$ is a probability measure since $\Ev[e^{\gamma S_n} ] = \Ev[ e^{\gamma \xi_1} ]^n = 1$. 
Also, since $e^{\gamma S_n} = \prod_{i=1}^n e^{\gamma \xi_i}$ it follows that a sequence $(\xi_1,\xi_2,\ldots)$ with distribution $Q$ is i.i.d.\ with mean $E_Q[ \xi_1] = \Ev[ \xi_1 e^{\gamma \xi_1} ]$. 
Since $x \log x$ is convex, it follows from Jensen's inequality that $\Ev[ \xi_1 e^{\gamma \xi_1} ] > \Ev[ e^{\gamma \xi_1} ] \log \Ev[ e^{\gamma \xi_1} ] = 0$ (note that the inequality is strict since the assumptions of the proposition imply that the distribution of $\xi_1$ is non-degenerate). 
Therefore, $E_Q[\xi_1] \in (0,\infty)$. 

For any $t\geq 0$ let $\tau(t) = \inf\{ n\geq 1: \, S_n \geq t \}$ be the stopping time for the first time the random walk $S_n$ goes above level $t$. Note that the event $\{ \sup_{n\geq 1} S_n \geq t \} = \{\tau(t) < \infty \}$, and since $\tau(t)$ is a stopping time this event only depends on $\xi_1,\xi_2,\ldots \xi_{\tau(t)}$. 
Therefore, applying the change of measure defined above we obtain that 
\[
 \Pv\left( \sup_{n\geq 1} S_n \geq t \right) 
= \Ev\left[ \ind{\tau(t) < \infty} \right] 
= E_Q\left[ e^{-\gamma S_{\tau(t)}} \ind{\tau(t) < \infty} \right] 
= E_Q\left[ e^{-\gamma S_{\tau(t)}} \right],
\]
where in the last equality we can drop the indicator of the event $\{\tau(t) < \infty\}$ since the fact that $E_Q[\xi_1] > 0$ implies that $Q(\tau(t) < \infty) = 1$. 

Since we are only considering the lattice case $\Pv( \xi_1 \in a \Z) = 1$ for some $a>0$, we need only to show that the limit 
\begin{equation}\label{agelimit}
 \lim_{k\ra\infty} E_Q\left[ e^{-\gamma (S_{\tau(ak)} - ak)}  \right] 
\end{equation}
exists. 
This will follow from results in renewal theory since $S_{\tau(ak)} - ak$ is the ``age'' of a renewal process at time $ak$ where the renewal increments have distribution $S_{\tau(a)}$. 
Note that Wald's identity implies that $E_Q[ S_{\tau(a)} ] = E_Q[ \xi_1 ] E_Q[\tau(a)] = \Ev[\xi_1 e^{\gamma \xi_1} ]  E_Q[\tau(a)]$ and thus if we show $E_Q[\tau(a)]<\infty$ then it will follow from standard results in renewal theory that $S_{\tau(ak)} - ak$ converges in distribution as $k\ra \infty$ (see \cite[Section 6.3]{lISP}).
To this end, note that
\[
 Q( \tau(a) > n ) \leq Q( S_n < a ) 
\leq e^{\gamma a/2} E_Q[ e^{-(\gamma/2) S_n} ] 
= e^{\gamma a/2} \Ev[ e^{ (\gamma/2) S_n} ]
= e^{\gamma a/2} \Ev[ e^{ (\gamma/2) \xi_1} ]^n.
\]
Since $u\mapsto \Ev[ e^{u \xi_1}]$ is convex as a function of $u$ and $\Ev[ e^{u \xi_1}]=1$ at $u=0$ and $u=\gamma$, 
it follows that $\Ev[ e^{ (\gamma/2) \xi_1} ] < 1$, and thus $\tau(a)$ has exponential tails under the measure $Q$. In particular, this implies that $E_Q[ \tau(a)]< \infty$ and so $S_{\tau(ak)} - ak$ converges in distribution and the limit \eqref{agelimit} exists. 
\end{proof}

The second result in this section concerns the behavior of the random walk prior to dropping below a certain level. 
Let
\[
 \nu(t) = \inf \{ n\geq 1: S_n \leq -t \}, \quad t \geq 0,
\]
be the first time the random walk $S_n$ drops below $-t$. 
\begin{prop}\label{grwxsum}
 Let $\Ev[ \xi_1] \in (-\infty,0)$, and assume that $\Ev[ e^{u \xi_1}] < \infty$ for some $u>0$. 
Then, there exists a constant $C <\infty$ such that
\[
 \Ev\left[ \sum_{n=0}^{\nu(t)-1} e^{-S_n} \right] \leq C e^t, \quad\forall t>0. 
\]
\end{prop}
\begin{proof}
For convenience of notation, let $\phi(t) = \Ev[ \sum_{n=0}^{\nu(t)-1} e^{-S_n}]$. We first show that $\phi(t) < \infty$ for all $t<\infty$. 
Since $e^{-S_n} < e^t$ for all $n<\nu(t)$, it follows that $\phi(t) \leq e^t \Ev[ \nu(t) ] $. 
To show that $\Ev[\nu(t)] < \infty$, note that for any $\d>0$
\[
 \Pv(\nu(t) > n) \leq \Pv(S_n > -t) \leq e^{\d t} \Ev[e^{\d S_n} ] = e^{\d t} \Ev[ e^{\d \xi_1}]^n. 
\]
Since $u\mapsto \Ev[ e^{u \xi_1}]$ is a convex function of $u$ with right derivative at $u=0$ equal to $\Ev[\xi_1] < 0$, there exists a $\d>0$ such that $\Ev[ e^{\d \xi_1}] < 1$. Thus, we can conclude that $\nu(t)$ has exponential tails and therefore $\Ev[\nu(t)] < \infty$. 
(Note that the above argument is enough to conclude that $\phi(t) \leq C e^{(1+\d)t}$ for some $C<\infty$ depending on $\d>0$.)

Since $\phi(t)$ is non-decreasing as a function of $t$, it is enough to prove that $\phi(k) \leq C e^k$ for all integers $k\geq 1$. 
By conditioning on $S_{\nu(k-1)}$,
\begin{align}
 \phi(k) &= \Ev\left[ \sum_{n=0}^{\nu(k-1)-1} e^{-S_n} \right] +  \Ev\left[ \sum_{n=\nu(k-1)}^{\nu(k)-1} e^{-S_n}  \ind{-k < S_{\nu(k-1)} \leq -k+1} \right] \nonumber \\
&\leq \phi(k-1) + e^k \Ev\left[ \sum_{n=\nu(k-1)}^{\nu(k)-1} e^{-(S_n-S_{\nu(k-1)}) } \ind{-k < S_{\nu(k-1)} \leq -k+1}  \right]\nonumber\\
&= \phi(k-1) + e^k \Ev\left[ \phi(k+S_{\nu(k-1)})  \ind{-k < S_{\nu(k-1)} \leq -k+1} \right] \nonumber\\
&\leq \phi(k-1) + e^k \phi(1), \label{phikub}
\end{align}
where the equality in the second to last line follows from the strong Markov property, and the last inequality follows from the fact that $\phi(t)$ is non-decreasing. It follows from \eqref{phikub} and induction that $\phi(k) \leq \frac{e^{k+1}-1}{e-1} \phi(1)$ for all $k\geq 1$. 
\end{proof}

\section{Exact RWRE calculations}\label{sec:exact}

Much of what is known about one-dimensional RWRE is due to the fact that certain probabilities and expectations of interest can be calculated explicitly. In preparation for the proofs of Theorems \ref{thm:st} and \ref{thm:srt} we will first review some of these formulas. 
We begin by introducing some notation that will help make these formulas more compact. 
Recall the definition of $\rho_x$ in \eqref{rhodef}, and for integers $i\leq j$ let 
\[
 \Pi_{i,j} = \prod_{x=i}^j \rho_x, \qquad 
R_{i,j} = \sum_{k=i}^j \Pi_{i,k}, \qquad \text{and}\qquad 
R_i = \sum_{k=i}^\infty \Pi_{i,k}.  
\]
With this notation we have the following formulas (the proofs of these formulas are easy Markov chain calculations and can be found in \cite{zRWRE}).  

\noindent\textbf{Hitting probabilities.}
Let the hitting times of the random walk be denoted by
\[
 T_x = \inf\{ n\geq 0: X_n = x \}, \quad x \in \Z. 
\]
Then, for any $a \leq x \leq b$ it is known that
\begin{equation}\label{htform}
 P_\w^x( T_a < T_b ) = \frac{ \Pi_{a,x-1} R_{x,b-1} }{R_{a,b-1}} \quad \text{and} \quad P_\w^x( T_a > T_b ) = \frac{ R_{a,x-1} }{R_{a,b-1}}.  
\end{equation}

\noindent\textbf{Quenched expectations of hitting times.}
For any environment $\w$ and any $x \in \Z$, 
\begin{equation}\label{ETform}
 E_\w^x\left[ T_{x+1} \right] = 1 + 2 \sum_{i\leq x} \Pi_{i,x}. 
\end{equation}
Of course, the sum in the above formula may possibly be infinite. However, if the RWRE is transient to the right then $E_P[\log \rho_i] < 0$ and thus the law of large numbers implies that $\Pi_{i,x} = \exp\{ \sum_{j=i}^x \log \rho_j \}$ decreases exponentially as $i\ra -\infty$ so that the sum in \eqref{ETform} converges almost surely. 
Similar reasoning shows that $E_\w^x[T_{x+1}] = \infty $ when $E_P[\log \rho_0] \geq 0$ (i.e., when the walk is recurrent or transient to the left). 
In this paper we will also need formulas for the expected values of hitting times to the left. From \eqref{ETform} and an obvious symmetry argument, one obtains
\begin{equation}\label{ETformleft}
 E_\w^x\left[ T_{x-1} \right] = 1 + 2 \sum_{i\geq x} \Pi_{x,i}^{-1}. 
\end{equation}
Before concluding this section, we note that the formulas for the limiting speed of the RWRE in \eqref{speedform} are derived from \eqref{ETform} and \eqref{ETformleft}. 
In particular, if the RWRE is transient to the right then it can be shown that $\lim_{n\ra\infty} X_n/n = \lim_{n\ra\infty} n/T_n = 1/\E[T_1]$. 
The formula for the speed in this case is then given by the fact that 
\begin{align*}
 \E[T_1] = E_P[ E_\w[T_1] ] &= 1 + 2 \sum_{i\leq 0} E_P[\Pi_{i,0}] \\
&= 1 + 2 \sum_{k= 0}^{\infty} E_P[\rho_0]^k 
=\begin{cases}
   \frac{1+E_P[\rho_0]}{1-E_P[\rho_0]} & \text{if } E_P[\rho_0] < 1 \\
   \infty & \text{if } E_P[\rho_0] \geq 1,
  \end{cases}
\end{align*}
where the second equality follows from \eqref{ETform} and the third equality follows from the fact that the environment was i.i.d.\ under the measure $P$.
The formula for the speed when the walk is transient to the left follows similarly from \eqref{ETformleft}.

\section{Quenched strong transience}\label{sec:QST}

In this section we will prove part \ref{qst} of Theorem \ref{thm:st}. 
The assumption that $E_P[\log \rho_0] \in (-\infty,0)$ implies that $P_\w( \mathfrak{R} <\infty) \geq 1-\w_0 > 0$ for $P$-a.e.\ environment $\w$. 
Therefore, we need only to show that $E_\w[  \mathfrak{R} \ind{ \mathfrak{R}<\infty} ]< \infty$. 
By conditioning of the first step of the walk, 
\begin{align}
 E_\w[  \mathfrak{R} \ind{ \mathfrak{R}<\infty} ] &= 1 + (1-\w_0) E_\w^{-1}[ T_0 \ind{T_0 < \infty} ] + \w_0 E_\w^1[ T_0 \ind{T_0 < \infty} ] \nonumber \\
&= 1 + (1-\w_0) E_\w^{-1}[ T_0 ] + \frac{\w_0R_1}{1+R_1} E_\w^1[ T_0 | \, T_0 < \infty ], \label{EwR1R}
\end{align}
where in the last equality we dropped the indicator from the first expectation since the walk is transient to the right, and we used that $P_\w^1( T_0 < \infty) = \frac{R_1}{1+R_1}$ from the quenched hitting time formulas in \eqref{htform}. 
It follows from the discussion following \eqref{ETform} that $E_\w^{-1}[T_0]$ is almost surely finite. 
Therefore, we need only to show that $E_\w^1[ T_0 | \, T_0 < \infty] < \infty$ for $P$-a.e.\ environment $\w$. 
To this end, note that conditioned on the event $\{T_{0} < \infty\}$ the law of the random walk until the stopping time $T_0$ is equal to that of a random walk in the environment $\tilde{\w} = \{\tilde{\w}_x \}_{x \in \Z}$ given by $\tilde{\w}_x = \w_x $ for $x\leq 0$ and 
\[
 \tilde{\w}_x = \frac{\w_x P_\w^{x+1}(T_{0} < \infty )}{P_w^x(T_{0} < \infty)} 
 = \frac{\w_x \frac{\Pi_{0,x} R_{x+1}}{R_0} }{ \frac{\Pi_{0,x-1} R_{x}}{R_0} } 
 = \frac{ \w_x R_{x+1}}{ 1 + R_{x+1} }, 
 \quad \text{ for } x \geq 1.
\]
(See \cite[page 78]{cgzLDP} for more details.)
Note that if we define $\tilde{\rho}_x = \frac{1-\tilde\w_x}{\tilde\w_x}$
and $\tilde{\Pi}_{i,j} = \prod_{x=i}^j \tilde{\rho}_x$, then we have that 
\[
 \tilde{\rho}_x = \frac{1+(1-\w_x) R_{x+1}}{\w_x R_{x+1}} = \frac{1+R_x}{R_{x+1}} =   \frac{1+R_x}{\rho_{x+1}(1+R_{x+2})} , \quad \forall x\geq 1, 
\]
and thus
\[
 \tilde{\Pi}_{i,j}
 = \frac{(1+R_i)(1+R_{i+1})}{\Pi_{i+1,j+1}(1+R_{j+1})(1+R_{j+2})}
 = \frac{(1+R_i)R_i}{\Pi_{i,j}(1+R_{j+1})R_{j+1} }, \quad \forall 1\leq i \leq j. 
\]
Using the explicit formula for quenched expectations of hitting times in \eqref{ETformleft}, we obtain that 
\begin{align}
 E_\w^1[ T_{0} \, | \, T_{0} < \infty] 
 = E_{\tilde{\w}}^1[ T_{0} ]
 &= 1 + 2 \sum_{n=1}^\infty (\tilde{\Pi}_{1,n} )^{-1} \nonumber \\
 &= 1 + 2 \sum_{n=1}^\infty \Pi_{1,n} \frac{(1+R_{n+1})R_{n+1}}{(1+R_1)R_1}.  \label{cqET}
\end{align}
To prove that the sum in \eqref{cqET} is finite, let $c_0 := -E_P[\log \rho_0] \in (0,\infty)$ and fix an $\e \in (0,c_0/5)$.  Then, the strong law of large numbers implies that for $P$-a.e.\ environment $\w$ there exists a finite integer $n_1(\w,\e)$ such that 
\begin{equation}\label{Pilbub}
 e^{-(c_0+\e)n} \leq \Pi_{1,n} \leq e^{-(c_0-\e)n}, \quad \forall n \geq n_1(\w,\e). 
\end{equation}
Note that this implies that $\Pi_{n+1,n+k} = \frac{\Pi_{1,n+k}}{\Pi_{1,n}} \leq \frac{e^{-(c_0-\e)(n+k)}}{e^{-(c_0+\e)n}} = e^{2\e n} e^{-(c_0-\e)k}$ for $n\geq n_1(\w,\e)$ and $k\geq 1$, and thus
\begin{equation}\label{Rub}
 R_{n+1}(1+R_{n+1}) \leq 
\left(1 + \sum_{k=1}^\infty \Pi_{n+1,n+k} \right)^2 \leq \frac{e^{4\e n}}{1-e^{-c_0+\e}}, \quad \forall n\geq n_1(\w,\e). 
\end{equation}
Therefore, \eqref{Pilbub} and \eqref{Rub} imply that 
$\Pi_{1,n}R_{n+1}(1+R_{n+1}) \leq \frac{e^{-(c_0-5\e)n}}{1-e^{-c_0+\e}}$ for all $n\geq n_1(\w,\e)$. 
Since we chose $\e<c_0/5$ this shows that the sum in \eqref{cqET} is almost surely finite.

\section{Averaged strong transience}\label{sec:AST}

We now turn to the results on strong transience under the averaged measure: Theorem \ref{thm:st}\ref{ast} and Theorem \ref{thm:srt}. 

\begin{proof}[Proof of Theorem \ref{thm:st}\ref{ast}]
Since $\P( \mathfrak{R} <\infty) \geq E_P[ 1-\w_0 ] > 0$, strong transience is equivalent to $\E[ \mathfrak{R} \ind{ \mathfrak{R}<\infty}] < \infty$. 
Averaging \eqref{EwR1R} with respect to the measure $P$ on environments we obtain that 
\begin{align}
 \E[ \mathfrak{R} \ind{ \mathfrak{R}<\infty}]
&= 1 + E_P\left[ (1-\w_0) E_\w^{-1}[T_0] \right] + E_P\left[ \frac{\w_0R_1}{1+R_1} E_\w^1[ T_0 | \, T_0 < \infty ] \right] \nonumber  \\
&= 1 + E_P[1-\w_0]E_P\left[E_\w^{-1}[T_0] \right]  + E_P[\w_0] E_P\left[  \frac{R_1}{1+R_1} E_\w^1[ T_0 | \, T_0 < \infty ] \right] \nonumber \\
&= 1 + E_P[1-\w_0]\E[T_1]  + E_P[\w_0] E_P\left[  \frac{R_1}{1+R_1} E_\w^1[ T_0 | \, T_0 < \infty ] \right], \label{ER1R}
\end{align}
where in the second equality we used that the environment $\{\w_x\}_{x\in\Z}$ is an i.i.d.\ sequence under the measure $P$ (note that $E_\w^{-1}[T_0]$ depends only on $\w_x$ with $x\leq -1$ and $E_\w^1[T_0 | \, T_0 < \infty]$ depends only on $\w_x$ with $x\geq 1$), 
and in the last equality we used the shift invariance of the enviroment under the distribution $P$. 
As noted in Section \ref{sec:exact} above, $\E[T_1] < \infty$ if and only if $E_P[\rho_0] < 1$. 
On the other hand, the formula for $E_\w^1[ T_0 | \, T_0 < \infty ]$ in \eqref{cqET} implies that 
\begin{equation}\label{rtequiv}
  E_P\left[  \frac{R_1}{1+R_1} E_\w^1[ T_0 | \, T_0 < \infty ] \right] < \infty 
\iff
E_P\left[ \sum_{n=1}^\infty \Pi_{1,n} \frac{(1+R_{n+1})R_{n+1}}{(1+R_1)^2} \right] < \infty. 
\end{equation}
Thus, to finish the proof of Theorem \ref{thm:st}\ref{ast} it remains only to show that that the right side of \eqref{rtequiv} holds when $E_P[\rho_0] < 1$. 
To accomplish this it is helpful to use the shift-invariance of the environment to re-write the sum in the following way. 
\begin{align*}
  E_P\left[ \sum_{n=1}^\infty \Pi_{1,n} \frac{(1+R_{n+1})R_{n+1}}{(1+R_1)^2} \right]
& = \sum_{n=0}^\infty E_P \left[ \Pi_{-n,0} \frac{(1+R_{1})R_{1}}{(1+R_{-n})^2} \right] \\
& = E_P\left[ R_1(1+R_1) \sum_{n=0}^\infty \frac{ \Pi_{-n,0} }{(1+R_{-n})^2} \right]. 
\end{align*}
For any $A>0$ let 
\[
 \pi(A) = \inf\left\{ n\geq 0: \Pi_{-n,0} \leq \frac{1}{A} \right\}
= \inf\left\{ n\geq 0: \sum_{i=0}^n \log \rho_{-i} \leq - \log(A) \right\}. 
\]
Note that $1+ R_{-n} = 1+R_{-n,0} + \Pi_{-n,0}R_1$. 
Then, 
\begin{align*}
 \sum_{n=0}^\infty \frac{ \Pi_{-n,0} }{(1+R_{-n})^2}
&\leq \sum_{n=0}^{\pi(R_1)-1} \frac{ \Pi_{-n,0} }{(1+R_{-n})^2}  + \sum_{n=\pi(R_1)}^\infty \Pi_{-n,0}   \\
&\leq \sum_{n=0}^{\pi(R_1)-1} \frac{ \Pi_{-n,0} }{(\Pi_{-n,0} R_1)^2}  + \Pi_{-\pi(R_1),0} \left( 1 + \sum_{n=\pi(R_1)+1}^\infty \Pi_{-n,-\pi(R_1)-1}  \right) \\
&\leq \frac{1}{R_1} \left\{ \sum_{n=0}^{\pi(R_1)-1} \frac{ 1 }{\Pi_{-n,0} R_1} +  \left( 1 + \sum_{n=\pi(R_1)+1}^\infty \Pi_{-n,-\pi(R_1)-1}  \right)  \right\} 
\end{align*}
Multiplying by $R_1(1+R_1)$ and taking expectations we get that
\begin{align} 
& E_P\left[ R_1(1+R_1) \sum_{n=0}^\infty \frac{ \Pi_{-n,0} }{(1+R_{-n})^2} \right] \nonumber  \\
&\qquad \leq E_P\left[ \frac{1+R_1}{R_1} \sum_{n=0}^{\pi(R_1)-1} \frac{ 1 }{\Pi_{-n,0}} \right] + E_P\left[ (1+R_1) \left( 1 + \sum_{n=\pi(R_1)+1}^\infty \Pi_{-n,-\pi(R_1)-1}  \right) \right] \nonumber \\
&\qquad = E_P\left[ \frac{1+R_1}{R_1} \sum_{n=0}^{\pi(R_1)-1} \frac{ 1 }{\Pi_{-n,0}} \right] + E_P\left[1+R_1\right] E_P\left[ 1 + \sum_{n=1}^\infty \Pi_{-n,-1}  \right] \nonumber\\
&\qquad = E_P\left[ \frac{1+R_1}{R_1} \sum_{n=0}^{\pi(R_1)-1} \frac{ 1 }{\Pi_{-n,0}} \right] + \left(E_P\left[ 1+R_1 \right]\right)^2, \label{R1decomp}
\end{align}
where in the second to last equality we used that the environment to the left of the origin $\{\w_x\}_{x\leq 0}$ is independent of $R_1$ and that $\pi(A)$ is a stopping time for the sequence $(\w_0,\w_{-1},\w_{-2},\ldots)$ for any $A>0$. 
To control the first expectation in \eqref{R1decomp}, note that it follows from Proposition \ref{grwxsum} that there exists a $C<\infty$ such that 
\[
 E_P\left[ \sum_{n=0}^{\pi(A)-1} \frac{1}{\Pi_{-n,0}} \right] = E_P\left[ \sum_{n=0}^{\pi(A)-1} e^{-\sum_{i=0}^n \log \rho_{-i}} \right] \leq C A, \quad \forall A<\infty. 
\]
Again, since the environment to the left of the origin is independent of $R_1$, by conditioning on $R_1$ we obtain that 
\[
 E_P\left[ \frac{1+R_1}{R_1} \sum_{n=0}^{\pi(R_1)-1} \frac{ 1 }{\Pi_{-n,0}} \right]
= E_P\left[ \frac{1+R_1}{R_1} E_P\left[  \sum_{n=0}^{\pi(R_1)-1} \frac{ 1 }{\Pi_{-n,0}} \biggl| \, R_1 \right] \right] 
\leq C E_P\left[1+R_1 \right]
\]

Combining the above results, we have shown that 
\begin{equation}\label{rtsc}
  E_P\left[ \sum_{n=1}^\infty \Pi_{1,n} \frac{(1+R_{n+1})R_{n+1}}{(1+R_1)^2} \right]
\leq C E_P\left[1+R_1 \right] + \left( E_P\left[1+R_1 \right] \right)^2. 
\end{equation}
Since $E_P[1+R_1] = (1-E_P[\rho_0])^{-1} < \infty $ when $E_P[\rho_0] < 1$,  
we have shown that expectation on the right side of \eqref{rtequiv} is finite if $E_P[\rho_0]<1$. 
\end{proof}

\begin{proof}[Proof of Theorem \ref{thm:srt}]
Part \ref{srtp1} of Theorem \ref{thm:srt} follows from Theorem \ref{thm:st}\ref{ast}, and thus we only need to prove part \ref{srtp2} of Theorem \ref{thm:srt}. 
Therefore, for the remainder of the proof we will assume that either $E_P[\rho_0] > 1$ or $E_P[\rho_0] = 1$ and $E_P[\rho_0 \log \rho_0] < \infty$. 
A calculation similar to \eqref{ER1R} shows that 
\begin{align*}
  \E[ \mathfrak{R} | \, X_1= 1, \mathfrak{R}<\infty ] 
&= \frac{ E_P\left[ \w_0 P_\w^1(T_0<\infty) (1+ E_\w^{1}[T_0| \, T_0<\infty])\right] }{ E_P[\w_0 P_\w^1(T_0<\infty)] } \\
&= 1 + \frac{ E_P\left[ \frac{R_1}{1+R_1} E_\w^{1}[T_0| \, T_0<\infty] \right] }{ \P^1(T_0<\infty) }. 
\end{align*}
Since $\P^1(T_0<\infty) > 0$, we need only to prove that $ E_P[ \frac{R_1}{1+R_1} E_\w^{1}[T_0| \, T_0<\infty] ] = \infty$, 
which by \eqref{rtequiv} is equivalent to showing that
\begin{equation}\label{infsum}
  \sum_{n=1}^\infty E_P\left[ \Pi_{1,n} \frac{R_{n+1}^2}{(1+R_1)^2} \right] = \infty. 
\end{equation}
To prove this we will need the following lemma which follows from the general random walk result in Proposition \ref{RWhp}. 
\begin{lem}\label{Rtail}
Assume that $E_P[\log \rho_0] < 0$ and that either 
\begin{enumerate}
 \item $E_P[\rho_0] > 1$
 \item or $E_P[\rho_0] = 1$ and $E_P[\rho_0 \log \rho_0] < \infty$. 
\end{enumerate}
Then, there exists a constant $C>0$ such that $P(R_1 \geq t) \geq \frac{C}{t}$ for all $t\geq 1$. 
\end{lem}
\begin{rem}
 It was shown by Kesten \cite{kRDE} that if the distribution of $\log \rho_0$ is non-lattice and $E_P[ \rho_0^\kappa ] = 1$ and $E_P[\rho_0^\kappa \log \rho_0 ] < \infty$ for some $\kappa > 0$, then $P(R_1 > t) \sim C t^{-\k}$ as $t\ra\infty$. If $E_P[\log \rho_0] < 0$ and $E_P[\rho_0^\k] = 1$ for some $\k \in (0,1]$, then $E_P[\rho_0] \geq 1$. Therefore, Lemma \ref{Rtail} gives rougher asymptotics than were obtained by Kesten, but under slightly less restrictive assumptions. 
\end{rem}
\begin{proof}[Proof of Lemma \ref{Rtail}]
First suppose that $E_P[\rho_0] <\infty$. 
Since the function $u \mapsto E_P[ \rho_0^u]$ is convex with right derivative equal to $E_P[\log\rho_0] < 0$ at $u=0$, then there exists a $\k>0$ such that $E_P[\rho_0^\kappa] = 1$. Moreover, $E_P[\rho_0^\kappa \log \rho_0] < \infty$ either due to the assumption of the lemma when $\kappa = 1$ or because $E_P[\rho_0] < \infty$ in the case when $\kappa < 1$. 
Then, it follows from Proposition \ref{RWhp} by letting $\xi_i = \log \rho_i$ and $\gamma =\kappa$ that  
\[
\liminf_{t\ra\infty} t^\kappa P(R_1 > t) \geq \liminf_{t\ra\infty} t^\kappa P\left( \sup_{n\geq 1} \Pi_{1,n} > t \right) 
= \liminf_{t\ra\infty} t^\kappa P\left( \sup_{n\geq 1} \sum_{i=1}^n \log \rho_i > \log t \right) > 0. 
\]
This completes the proof of the lemma in all cases except when $E_P[\rho_0] = \infty$.
If $E_P[\rho_0] = \infty$ then for $M<\infty$ large enough $E_P[ \rho_0 \wedge M ] \in (1,\infty)$, and since 
\[
 R_1 = \sum_{j=1}^\infty \prod_{i=1}^j \rho_i \geq \sum_{j=1}^\infty \prod_{i=1}^j (\rho_i \wedge M) =: R_1^{(M)}, 
\]
it follows from the first part of the proof that $P(R_1 > t) \geq P(R_1^{(M)} > t) \geq C/t$ for $t\geq 1$. 
\end{proof}

Returning now to the proof of Theorem \ref{thm:srt}, let $\mathcal{F}_n = \s( \w_x: \, x\leq n)$ be the $\s$-field generated by the environment to the left of $x=n$. 
Using the fact that $1+R_1 = 1 + R_{1,n} + \Pi_{1,n} R_{n+1}$,
\begin{align*}
 E_P \left[  \frac{\Pi_{1,n} R_{n+1}^2}{(1+R_1)^2} \right] 
 &\geq \frac{1}{4} E_P\left[  \Pi_{1,n}^{-1} \ind{\Pi_{1,n} R_{n+1} \geq 1 + R_{1,n}} \right]\\
&= \frac{1}{4} E_P\left[  \Pi_{1,n}^{-1} P\left( R_{n+1} \geq \frac{1 + R_{1,n}}{\Pi_{1,n} } \biggl| \, \mathcal{F}_n \right) \right].
\end{align*}
Since the assumptions of Lemma \ref{Rtail} are satisfied and $R_{n+1}$ is independent of $\mathcal{F}_n$, the conditional probability in the last line above is bounded below by $C \Pi_{1,n}/(1+R_{1,n})$ for some $C>0$ (note that we used $(1+R_{1,n})/\Pi_{1,n} \geq (1 + \Pi_{1,n})/\Pi_{1,n} > 1$ here).
Therefore, we can conclude that 
\[
 E_P \left[  \frac{\Pi_{1,n} R_{n+1}^2}{(1+R_1)^2} \right] \geq \frac{C}{4} E_P\left[ \frac{1}{1+R_{1,n}} \right] 
\geq \frac{C}{4} E_P\left[ \frac{1}{1+R_1} \right] > 0. 
\]
Clearly this implies that \eqref{infsum} holds, and thus this finishes the proof of the second part of Theorem \ref{thm:srt}. 
\end{proof}

\bibliographystyle{alpha}
\bibliography{RWRE}


\end{document}